\documentclass{amsart}
\pdfoutput=1
\allowdisplaybreaks
\usepackage{mathpazo}
\usepackage{amsmath,amsfonts,amssymb}
\usepackage{amsrefs}
\usepackage{amsthm}
\usepackage{latexsym,amsmath,amssymb,amsfonts}
\usepackage{rotating}
\usepackage{mathrsfs}
\usepackage{amscd}
\usepackage{hyperref} %% pdflatex
\usepackage{euscript}
\usepackage{hhline}
\usepackage{graphicx,epstopdf}
\usepackage{epsfig}
\usepackage{xcolor}
\usepackage[all,color]{xy}
\usepackage{tikz}
\usetikzlibrary{cd}
\usetikzlibrary{matrix}
\usepackage{spectralsequences}

\newlength{\fighskip} \fighskip=2pt
\newlength{\figvskip} \figvskip=3pt

\usepackage{hyperref}

\advance\textheight by \topskip
\numberwithin{equation}{section}

\renewcommand{\epsilon}{\varepsilon}

\newcommand{\R}{\mathbb{R} }
\newcommand{\C}{\mathbb{C} }
\newcommand{\Z}{\mathbb{Z} }

\newcommand{\U}{{\mathcal{U}}}

\newcommand{\pp}[2]{\frac{\partial #1}{\partial #2}}  %\pp{a}{b} not \pp ab

\newcommand{\LR}[1]{\left ( #1 \right )}
\newcommand{\Lr}[1]{\left [ #1 \right ]}
\newcommand{\lr}[1]{\left\{ #1\right \}}
\newcommand{\lR}[1]{\left < #1\right >}
\newcommand{\ali}[1]{$$\begin{aligned} #1 \end{aligned}$$}

\newcommand{\inv}{^{-1}}

\newcommand{\lra}{ \longrightarrow}
\newcommand{\lmt}{ \longmapsto}

\newcommand{\zz}{{\mathbf{z}}}
\newcommand{\ww}{{\mathbf{w}}}
\newcommand{\xx}{{\mathbf{x}}}
\newcommand{\xxi}{\boldsymbol \xi}

\newcommand{\OO}{{\mathcal O}}

\newcommand{\DD}{{\mathbb D}}

\renewcommand{\emptyset}{\varnothing}

\newcommand{\bp}{{\Bar{\partial}}}

\def\XXint#1#2#3{{\setbox0=\hbox{$#1{#2#3}{\int}$}
		\vcenter{\hbox{$#2#3$}}\kern-.5\wd0}}

\DeclareMathOperator{\Conf}{Conf}

\newcommand{\Hol}{\OO^{hol}}
\newcommand{\HolX}[1]{{\OO^{hol}_{#1}}} 

\DeclareMathOperator{\Id}{Id}
\renewcommand{\Im}{\mathrm{Im} \;\!}
\DeclareMathOperator{\Ker}{Ker}

\theoremstyle{plain}
\newtheorem{thm}{Theorem}[section]
\newtheorem{lem}[thm]{Lemma}
\newtheorem{prop}[thm]{Proposition}
\newtheorem{cor}[thm]{Corollary}

\theoremstyle{definition}
\newtheorem{defn}[thm]{Definition}

\theoremstyle{remark}
\newtheorem{rem}[thm]{Remark}

\allowdisplaybreaks[4]  %%%%%%%%%%% choose 1-4, where 4 is the strongest desire to breack

\title{\textbf On the Dolbeault Cohomology of Configuration Spaces}
\author{Peng Yang}
  \address{
P.~ Yang: Department of Mathematical Sciences, Tsinghua University, Beijing, China
}
\email{yang-p20@mails.tsinghua.edu.cn}

\begin{document}
\begin{abstract}  
We compute the Dolbeault cohomology ring of the configuration spaces of $\mathbb{C}^n$ and construct a spectral sequence that converges to the Dolbeault cohomology ring of the configuration spaces of an arbitrary complex manifold.
\end{abstract}

\maketitle
  
\tableofcontents

%%%%%%%%%%%%%%%%%%%%%%%%%%%%%%%%%%%%%%%%%%%
%                                         %
%                                         %
%                                         %
%                                         %
%                                         %
%%%%%%%%%%%%%%%%%%%%%%%%%%%%%%%%%%%%%%%%%%%

\section{Introduction}
Let $\Conf_m(\R^n)$ be the configuration space of $m$ distinct ordered points in $\R^n$. Via the Gauss map 
\ali{
\gamma_{ab}:\qquad \Conf_m(\R^n)         &\lra S^{n-1}\\ 
  (\xx^1,\cdots,\xx^m) &\lmt    \frac{\xx_b-\xx_a}{||\xx_b-\xx_a||}
}
the volume form on $S^{n-1}$ pulls back to a degree $n-1$ form $\theta_{ab}$ on $\Conf_m(\R^n)$. As is well known, the cohomology ring $H^\bullet(\Conf_m(\R^n))$ is freely generated by the classes $\theta_{ab}$, subject to the relations $\theta_{ab}=(-1)^n \theta_{ba}$, $\theta_{ab}^2=0$,  and the Arnold relation $\theta_{ab}\theta_{bc}+\theta_{bc}\theta_{ca}+\theta_{ca}\theta_{ab}=0$. For a review, see \cite{cohomology-conf}.
 
Let $M$ be an oriented manifold of dimension $n$. To compute $H^\bullet(\Conf_m(M))$, Totaro \cite{cohomology-conf-mfd} studied the Leray spectral sequence associated to the inclusion $\Conf_m(M) \hookrightarrow M^m$. The spectral sequence converges to $H^\bullet(\Conf_m(M))$, and its $E_2$ page, as a bigraded algebra, is the quotient of $H^\bullet(M^m)[\theta_{ab}]$ by the above three relations and an additional relation $p_a^*(\tau) \theta_{ab} = p_b^*(\tau) \theta_{ab}$ for $\tau\in H^{0,\bullet}_\bp(M)$, where $p_a$ is the $a$-th projection $M^m\to M$. 
The differential is given by
$$
d\theta_{ab}=p_{ab}^*\Theta 
$$
where $p_{ab}$ is the projection $M^n\to M^2$ onto the $a$-th and the $b$-th coordinates and $\Theta $ is the class of the diagonal in $H^{n}(M^2)$. When $M$ is the underlying space of a smooth projective variety over $\C$, the spectral sequence degenerates after the first nontrivial differential.

In this paper, we study Dolbeault cohomology rings of  configuration spaces of complex manifolds. 
%These spaces are often infinite-dimensional and are equipped with suitable topologies. 
The Dolbeault cohomology of configuration spaces is not only of intrinsic mathematical interest but also a window into the algebraic structure of holomorphic field theories, known as holomorphic factorization algebras \cite{kevin-owen}. 
%Factorization algebras, which provide a geometric axiomatization of the algebra of observables in quantum field theory, generalize the notion of vertex algebras to higher dimensions. 
Our explicit description of $H^{0,\bullet}_\bp(\Conf_m(\C^n))$ in Theorem \ref{Dol-ring-structure-intro} can be interpreted as giving a complete characterization of the local operators in a free holomorphic field theory.  

Our first step is the computation of  
$H^{0,\bullet}_{\bp}(\C^n-0)$. 
There exists a  class $\omega$ of degree $(0,n-1)$ that  represents the inverse of $\bp$ applied to the Dirac delta distribution at $0$; hence, $\omega$ is $\bp$-closed. Moreover, since %holomorphic differential operators
$\partial_i:=\frac{\partial}{\partial z_i}$ commutes with $\bp$, applying $\partial_i$ any number of times to $\omega$ still yields a $\bp$-closed form. These forms constitute  a Schauder basis for $H^{0,n-1}_{\bp}(\C^n-0)$, and we have
\begin{prop}[See Proposition \ref{Dolbeault-punctured-disk}]
$$
H^{0,\bullet}_{\bp}(\C^n-0)=
\begin{cases}
\Hol(\C^n) & \bullet=0 \\
\overline{\C[\partial_1,\cdots,\partial_n]\omega}   & \bullet=n-1 \\
0             & \text{else}
\end{cases}
$$
\end{prop}

Second, we compute $H^{0,\bullet}_\bp (\Conf_m(\C^n))$. Let $\pi: \Conf_m(\C^n) \to \Conf_{m-1}(\C^n)$ be the projection onto the first $m-1$ coordinates. It is not a holomorphic fiber bundle because the fibers are generally not biholomorphic; however, the Leray spectral sequence still applies. 
Let $\phi_{ab}$ be the map
\ali{
\phi_{ab}: \Conf_m(\C^n) &\lra \C^n-0 \\
  (\zz^1,\cdots,\zz^m)       &\lmt \zz^a-\zz^b
}
and denote $\omega_{ab}=\phi_{ab}^*\omega$.  
The Leray spectral sequence shows $H^{0,\bullet}_\bp (\Conf_m(\C^n))$ is generated by holomorphic functions and holomorphic derivatives of $\omega_{ab}$ with relations analogous to those in \cite{cohomology-conf-mfd}. More precisely, we have

\begin{thm}[See Theorem \ref{Dol-ring-structure-v2} and Corollary \ref{Dol-ring-structure-C}]\label{Dol-ring-structure-intro}
The graded commutative topological ring $H^{0,\bullet}_\bp(\Conf_m(\C^n))$ is generated by holomorphic functions and holomorphic derivatives of $\omega_{ab}$, subject to the following relations 
%$\omega_{ab}=(-1)^n \omega_{ba}$, $f(\vec \partial_b)\omega_{ab}  g(\vec \partial_b)\omega_{ab}=0$, (holomorphic derivatives of) $\omega_{ab}\omega_{bc}+\omega_{bc}\omega_{ca}+\omega_{ca}\omega_{ab}=0$, and $(f(\zz^a)-f(\zz^b))g(\vec \partial_b)\omega_{ab}=0$.
\begin{itemize}
\item $\omega_{ab}=(-1)^n \omega_{ba}$
\item $f(\vec \partial_b)\omega_{ab}  g(\vec \partial_b)\omega_{ab}=0$, for $n\ge2$
\item (holomorphic derivatives of) $\omega_{ab}\omega_{bc}+\omega_{bc}\omega_{ca}+\omega_{ca}\omega_{ab}=0$
\item $(f(\zz^a)-f(\zz^b))g(\vec \partial_b)\omega_{ab}=0$
\end{itemize} 
For $n=1$, $H^{0,\bullet}_\bp (\Conf_m(\C)) = \Hol(\Conf_m(\C))$, and the cohomology vanishes in positive degree
\end{thm}
%Since the holomorphic cotangent bundle of $\Conf_m(\C^n)$ is trivial, we also get the structure of $H^{\bullet,\bullet}_{\bp}(\Conf_m(\C^n))$.

Finally, For a   complex manifold $X$, there is a spectral sequence which converges to  $H^{\bullet,\bullet}_\bp(\Conf_m(X))$. 
\begin{thm}[See Theorem \ref{conf-mfd-cohomology}]
Let $X$ be a complex manifold with dimension $n\ge2$. The inclusion $\Conf_m(X)\hookrightarrow X^m$ determines a Leray spectral sequence which converges to $H^{\bullet,\bullet}_\bp(\Conf_m(X))$ as a graded commutative topological ring. The $E_2$ term is a sheaf of rings with underlying group
$$
\oplus_{k=0}^{m-1}\oplus_{\substack{1< b_1<\cdots<b_{k}\le m\\ a_1<b_1,\cdots,a_k<b_k}} 
H^\bullet \LR{X^m, \Omega^{\bullet,0}_{X^m} \otimes_{\HolX{X^m}} G_{a_1b_1} \otimes_{\HolX{X^m}} \cdots \otimes_{\HolX{X^m}} G_{a_kb_k}}.
$$   
Here the sheaf $G_{ij}$, supported on the diagonal $\Delta_{ij}$, is constructed by applying holomorphic differential operators to the inverse canonical sheaf of $\Delta_{ij}$.  
The differential $d_n$ is the  Gysin map.  % (see Section \ref{ss-confX}). 
\end{thm} 
As an example, we calculated the spectral sequence explicitly for compact tori.

%The notion of regularized integrals on $\Conf_m(\C^n)$ will be developed by Minghao Wang and the   author in future.  

\subsection{Notations} 
\begin{itemize} 
\item %Unless otherwise specified, $n\ge2$. 
Vectors are denoted in boldface. For example, a point in $\C^n$ is denoted by $\zz=(z_1,\cdots,z_n)$. 
\item The configuration space of $m$ points in $\C^n$ is defined by 
$$
\Conf_m(\C^n)=\{(\zz ^1,\cdots,\zz ^m)\in \C^n\times\cdots\times\C^n|  \zz^i\ne\zz^j \text{ for } i\ne j\}.
$$
Here $\zz^i=(z^i_1,\cdots,z^i_n)\in\C^n$.
\item The space of smooth (resp. holomorphic) $(p,q)$-forms on an open subset $U$ of a complex manifold is denoted by $A^{p,q}(U)$ (resp. $\Omega^{p,q}(U)$). The space of holomorphic functions on $U$ is denoted by $\Hol(U)$.  The structure sheaf of $U$ is denoted by $\HolX{U}$.
\item  $A^{p,q}(U)$ (resp. $\Omega^{p,q}(U)$) is equipped with the $C^\infty$ topology, which makes it a nuclear space. The closure of a subspace of  $A^{p,q}(U)$ is always taken with respect to the $C^\infty$ topology. For example, the closure of the space of polynomial functions on $U$ is the space of power series that converge on $U$.  
\item $\bp:A^{p,q}(U)\to A^{p,q+1}(U)$ is the Dolbeault differential. The space of $(p,q)$-th Dolbeault cohomology on $U$ is denoted by $H^{\bullet,\bullet}_{\bp}(U)$. $H^{\bullet,\bullet}_{\bp}(U)$ is equipped with the quotient topology. 
\item We use $\otimes$ without subscript to denote tensor product of nuclear spaces over $\C$, which coincides with both the injective tensor product and projective tensor product. We have $\Hol(U)\otimes \Hol(V) =\Hol(U\times V)$.
\item A   hat $(\,\widehat{\cdot}\,)$ over a symbol indicates that it is omitted.  % For example, $(z_1,\cdots, \widehat{z}_i,\cdots,z_n):=(z_1,\cdots, z_{i-1},z_{i+1},\cdots, z_n)$. 
\item Denote $\partial_i=\partial_{z_i}=\pp{}{z_i}$. A holomorphic derivative acts on components of differential forms, i.e., acts as Lie derivatives.
%A holomorphic derivative $\partial_{z_i}$ (resp. $\partial_{z_i^j}$)  acting a differential form always means the Lie derivative $\mathcal L_{\partial_{z_i}}$ (resp. $\mathcal L_{\partial_{z_i^j}}$) associated with it.
\item A function of  vectors depends on the components of those vectors. For example, $f(\zz^{i_1},\cdots,\zz^{i_k})$ is a holomorphic function, and $f(\vec \partial_{b_1},\cdots,\vec \partial_{b_k})$ is a power series of holomorphic derivatives.
\end{itemize}

\noindent \textbf{Acknowledgment}. The author would like to thank Professor Si Li for guidance, and thank Minghao Wang, Xiaoxiao Yang and Tianqing Zhu for helpful communications.  

\section{Preliminaries}
In this section, we briefly review some constructions in algebraic topology and refer the reader to \cite{Bott-Tu} for details.  

%\subsection{\v{C}ech-Dolbeault bicomplex}
\noindent \textbf{\v{C}ech-Dolbeault bicomplex.} 
%In this section we introduce the notion of \v{C}ech-Dolbeault bicomplex. See  for example \cite{Griffiths-Harris}.
Let $X$ be a complex manifold with an open cover  $\U=\lr{U_i}_{i\in I}$. Define
$$
C^r(\U,A^{p,q}) = \prod_{i_0<\cdots<i_r} A^{p,q}(U_{i_0 \cdots i_r})
$$
%where we have used an order on the index set, and 
where $U_{i_0 \cdots i_r}=U_{i_0}\cap \cdots \cap U_{i_r}$. 
%The \v{C}ech differential $\delta$ and $\bp$ make $C^\bullet(\U,A^{p,\bullet})$ a bicomplex, which is called the  $p$-th \v{C}ech-Dolbeault bicomplex.
Denote by $\delta$ the \v{C}ech differential. The triple  $(C^\bullet(\U,A^{p,\bullet}),\delta,\bp)$ %equipped with the \v{C}ech differential $\delta$ and the Dolbeault differential $\bp$ 
is called the  $p$-th \v{C}ech-Dolbeault bicomplex.
Denote its  (hyper-)cohomology   %$A^{p,q}(\U)$ 
by $H^{p,q}(\U)$.
\begin{prop}\label{Cech-resolution-iso}
The restriction map $A^{p,q}(X)\to C^0(\U,A^{p,q})$ induces an isomorphism  
$$
H^{p,q}_\bp(X) \simeq H^{p,q}(\U).
$$
\end{prop}

%\subsection{Leray spectral sequence}
\noindent \textbf{Leray spectral sequence.}  
% see Griffths pp 463 
% Bott tu pp.179 
% https://en.wikipedia.org/wiki/Leray_spectral_sequence
% Dimca, Alexandru (2004). Sheaves in Topology.
Let $f:X\to Y$ be a continuous map of smooth manifolds. 
Suppose there is a sufficiently fine open cover $\U=\lr{U_i}$ of $Y$. Then $f\inv\U:=\lr{f\inv(U_i)}$ is an open cover of $X$, which can be used to compute cohomology of a sheaf $\mathcal F$ on $X$. 
Consider an acyclic resolution $0 \overset{\phantom{d}}{\to} \mathcal F  \overset{\phantom{d}}{\to} \mathcal F^0 \overset{d}{\to} \mathcal F^1 \overset{d}{\to}   \cdots$ and form a bicomplex $(C^\bullet(f\inv\U,\mathcal F^\bullet),\delta,d)$. %When the cover $\U$ is nice enough,  
$H^\bullet_{\delta+d}(C^\bullet(f\inv\U,\mathcal F^\bullet))$ computes the sheaf cohomology $H^\bullet(X,\mathcal{F})$. Moreover, there is a    spectral sequence associated with the bicomplex which converges to $H^\bullet(X,\mathcal F)$ with $E_2$ page
\ali{
E_2^{p,q}= H^p(f\inv\U, \mathcal{H}^q) 
}
where $\mathcal{H}^q$ is  the presheaf on $Y$ sending $U\mapsto H^q(f\inv\U,\mathcal{F})$, whose sheafification is $R^qf_*\mathcal{F}$, the $q$-th higher direct image of $\mathcal{F}$. % in the category of sheaves of abelian groups
In general, there is a sheaf theoretical version of the spectral sequence with 
$$
E_2^{p,q}=H^p(R^qf_*\mathcal{F}) \Longrightarrow H^{p+q}(X,\mathcal{F}).
$$

\section{Dolbeault cohomology ring of $\Conf_m(\C^n)$}

\subsection{Dolbeault cohomology group of punctured polydisks}\label{punctured-polydisks}
%, and $\Conf_m(\DD^n)=\lr{\LR{z_1,\cdots,z_m}\subset \DD^n|z_i\ne z_j \text{ for }i\ne j}$ be the configuration space of $m$ points in the polydisk $\DD^n$, $n\ge2$. In this section we calculate the Dolbeault cohomology $H^{0,\bullet}_{\bp}(\Conf_m(\DD^n))$.
%In this section we prove
Let $\omega_{BM}$ be the Bochner-Martinelli kernel
$$
\omega_{BM}(\zz,\bar \zz)=
\frac{(n-1)!}{(2\pi i)^{n}}\frac{1}{|\zz|^{2n}}\sum_{i=1}^n \bar z_i d\bar z_1 dz_1 \cdots \widehat{d\bar z_i} dz_i \cdots d\bar z_n dz_n  \in \Omega^{n,n-1 }(\C^n- 0)
%(-1)^{\frac {n(n-1)}{2}+1}(n-1)!(2\pi i)^{-n}|\zz|^{-2n}\sum_i (-1)^i \bar z_i d\bar z_1 \cdots \widehat{d\bar z_i} \cdots d\bar z_n  \in \Omega^{0,n-1 }(\C^n- 0)
$$ 
which serves as an integral kernel for the inverse of $\bp$. Specifically, let $U\subset\C^n$ be a bounded open set with smooth boundary and $f%(\zz,\bar \zz)
:\overline U\to \C$ be a smooth function, then for $\zz\in U$,  
$$
f(\zz,\bar \zz) = \int_{\xxi\in \partial U} f(\xxi,\bar\xxi) \omega_{BM}(\xxi-\zz,\bar\xxi-\bar\zz) -\int_{\xxi\in  U} \bp f(\xxi,\bar\xxi) \omega_{BM}(\xxi-\zz,\bar\xxi-\bar\zz).
$$
Due to  the uniqueness of the Bochner kernel as an inverse of $\bp$ in different coordinate systems, we have 
\begin{prop}\label{BM-coor-change}
Let $\zz\to \ww=\ww(\zz)$ be a coordinate transformation that fixes $0\in \C^n$. Then 
\begin{itemize}
\item For $n=1$, $\omega_{BM}(w,\bar w)=\frac{dw}{w}=\frac{dz}{z}+$ holomorphic $1$-form on $\C$.
\item For $n\ge 2$, $\omega_{BM}(\ww,\bar\ww)=\omega_{BM}(\zz,\bar \zz)+\bp$-exact term.
\end{itemize}
\end{prop}
\iffalse
\begin{proof}
For $n=1$, the result can be directly computed. 
For $n\ge2$, we use the \v Cech resolution above, where $\omega_{BM}(\zz,\bar \zz)$ is homotopic to the Cauchy kernel $\gamma(\zz)=(2\pi i)^{-n}\frac {dz_1\cdots dz_n}{z_1\cdots z_n}\in 
\Hol((\C-0)^n)$. We only need to show  Cauchy kernel are unique up to a $\delta$-exact term. 
we can assume only the first coordinate of $\zz$ changes, then $\gamma(\ww)-\gamma(\zz)$ is homotopic to $(2\pi i)^{-n}\frac {dz_2\cdots dz_n}{z_2\cdots z_n}\cdot$ some holomorphic $1$-form, which is $\delta$-exact. 
\end{proof}
\fi

Let us  write $\omega_{BM}=dz_1\cdots dz_n \omega$, where
$$
\omega(\zz,\bar \zz)=(-1)^{\frac {n(n-1)}{2}+1}\frac{(n-1)!}{(2\pi i)^{n}}\frac{1}{|\zz|^{2n}}\sum_{i=1}^n (-1)^i \bar z_i d\bar z_1 \cdots \widehat{d\bar z_i} \cdots d\bar z_n  \in \Omega^{0,n-1 }(\C^n- 0).
$$
$\omega$ is a $\bp$-closed form on $\C^n-0$. Since $\partial_i:=\frac{\partial}{\partial z_i}$ commutes with $\bp$, applying $\partial_i$ any number of times to $\omega$ still yields a $\bp$-closed form. These forms span a subspace of smooth forms
\ali{
\C[\partial_1,\cdots,\partial_n]\omega
=\bigoplus_{i_1,\cdots,i_n\ge 0} \C\cdot \partial^{i_1}_{1}\cdots \partial^{i_n}_{n}\omega 
\subset A^{0,n-1}(\C^n-0).
}
In this paper, the closure of a space of differential forms is always taken with respect to the $C^\infty$ topology (see Notations). The closure of $\C[\partial_1,\cdots,\partial_n]\omega$ is denoted by $\overline{\C[\partial_1,\cdots,\partial_n]\omega}$, which is clearly a subspace of $\Ker \bp$. We also denote 
\ali{
\C[[\partial_1,\cdots,\partial_n]]\omega
=\prod_{i_1,\cdots,i_n\ge 0} \C\cdot \partial^{i_1}_{1}\cdots \partial^{i_n}_{n}\omega.
}
 
In practice, we will also consider polydisks. A polydisk $\DD^n:=\prod_{i=1}^n \DD_i$ is a product of disks $\DD_i=\lr{z_i\in \C | |z_i|<r_i}$, $0<r_i\le\infty$. The following propositions states $\partial^{i_1}_{1}\cdots \partial^{i_n}_{n}\omega$ forms a Schauder basis for  $H^{0,n-1}_{\bp}(\DD^n-0)$.

\begin{prop}\label{Dolbeault-punctured-disk}
$
H^{0,\bullet}_{\bp}(\DD^n-0)=
\begin{cases}
\Hol(\DD^n) & \bullet=0 \\
\overline{\C[\partial_1,\cdots,\partial_n]\omega}   & \bullet=n-1 \\
0             & \text{else}
\end{cases}.
$ 
\iffalse
\\
More explicitly, let $\overline{\C[\partial_1,\cdots,\partial_n]\omega}$ be the closure of the subspace 
$\C[\partial_1,\cdots,\partial_n]\omega
%=\oplus_{I\in \Z^n_{\ge 0}} \partial^I\omega
=\oplus_{i_1,\cdots,i_n\ge 0} \partial^{i_1}_1\cdots \partial^{i_n}_n\omega 
\subset A^{0,n-1}(\C^n-0)$
, then $\overline{\C[\partial_1,\cdots,\partial_n]\omega} \subset \Ker \bp_{n-1}$, and the natural map $\overline{\C[\partial_1,\cdots,\partial_n]\omega} \to H^{0,n-1}_{\bp}(\DD^n-0)$ which sends a closed form to its $\bp$-cohomology class is an isomorphism.
\fi
\end{prop}

\begin{proof}
By Hartogs' extension theorem, we have $H^{0,0}_{\bp}(\DD^n-0)=\Hol(\DD^n)$. Moreover, the second Riemann continuation theorem (see for example \cite{Stein}, Chapter V) %page 133 %to the   codimension $n$ subset $\lr{0}\subset \DD^n$.
shows $H^{0,\bullet}_{\bp}(\DD^n-0)=0$ for $0<\bullet<n-1$.

In the remainder of the proof, we compute  $H^{0,n-1}_{\bp}(\DD^n-0)$ by \v Cech cohomology. Let $U_i=\{(z_1,\cdots,z_n)\in\DD^n|z_i\neq 0\}$, then $\U=\{U_1,\cdots,U_n\}$ is an open cover of $\DD^n-0$. 
We need to compute Dolbeault cohomology of the intersection $U_{i_0 \cdots i_r}=U_{i_0}\cap \cdots \cap U_{i_r}$. Notice that $U_{i_0 \cdots i_r}$ is the product of some disks and some punctured disks, and any finite product of Stein spaces is again Stein. As a result, $H^{0,\bullet>0}_{\bp}(U_{i_0 \cdots i_r})=0$ and $\Hol(U_{i_0 \cdots i_r})$ is a tensor product of holomorphic functions on $\DD_j$ or $\DD_i-0$  
\begin{equation*}\label{hol-decomposition-1}
\Hol(U_{i_0 \cdots i_r})
={\bigotimes_{j\notin \lr{i_0,\cdots,i_r}}\Hol(\DD_j)} \,\otimes {\bigotimes_{i\in \lr{i_0,\cdots,i_r}}\Hol(\DD_i-0)}. \tag{*}
\end{equation*}
We claim that 
\begin{equation*}\label{hol-decomposition-2}
\Hol(\DD_i-0)=\Hol(\DD_i)\oplus \C\{z_i\inv\}, \tag{**}
\end{equation*}
where $\C\{z_i\inv\}$ is the space of power series $\sum_{k>0} a_{k} z_i^{-k}$ that converges on $\C-0$. 
%$\C\{z_i\inv\}=\lr{\text{series } \sum_{k>0} a_{k} z_i^{-k} \text{ that converges on } \C-0}$. 
Indeed, let $\DD-0$ be a punctured disk with radius $r$, and $f\in \Hol(\DD-0)$ be a holomorphic function on $\DD-0$ with Laurent series $\sum_{k\in\Z} a_kz^k$. The non-negative part $\sum_{k\ge0} a_kz^k$ converges on $\DD$, so it belongs to $\Hol(\DD_i)$. To see the negative part $\sum_{k>0} a_kz^{-k}$ converges on $\C-0$, let $w=z\inv$, then the series converges on $\lr{w\in\C  \middle| |w|>\frac{1}{r}}$, so its radius of convergence is infinite when expressed in $w$.

Using (\ref{hol-decomposition-1}, \ref{hol-decomposition-2}), we can explicitly compute the \v Cech cohomology $H^{\bullet}(\U,\Hol)$. The $(n-1)$-th term is
$$
H^{n-1}(\U,\Hol) =
\text{Coker}\!\;(\oplus_{i=1}^n \Hol(U_{1\cdots \hat i \cdots n})\overset{\delta}{\rightarrow} \Hol(U_{1\cdots  n}))
=\C\{z_1\inv\}\otimes\cdots\otimes \C\{z_n\inv\}.
$$
The key observation is that $H^{n-1}(\U,\Hol)$ is generated by the Cauchy kernel 
$$
\gamma=(2\pi i)^{-n}\frac {1}{z_1\cdots z_n}\in 
\Hol((\C-0)^n).
$$
Denote 
$$
\C[\partial_1,\cdots,\partial_n]\gamma := \bigoplus_{i_1,\cdots,i_n\ge 0} \C\cdot \partial^{i_1}_{1}\cdots \partial^{i_n}_{n}\gamma \subset \Hol ((\C-0)^n),
$$
then $H^{n-1}(\U,\Hol)=\overline{\C[\partial_1,\cdots,\partial_n]\gamma}$. 

Recall the Dolbeault isomorphism can be realized as a composition of isomorphisms (see for example \cite{Griffiths-Harris}, Chapter 5) %Section 5.1. pp. 652.
$$
H^{n-1}(\U,\Hol) 
\simeq H^{n-2}(\U,\Omega^{0,1}) 
\simeq \cdots 
\simeq H^{1}(\U,\Omega^{0,n-2}) 
\simeq \Gamma(\U,\Omega^{0,n-1})/\bp \Gamma(\U,A^{0,n-2})
=H^{0,n-1}_{\bar\partial}(\DD^n-0)
$$
where each isomorphism is induced by the exact sheaf sequence 
$$
0\lra \mathfrak E^{0,\bullet}\lra \mathfrak A^{0,\bullet}\overset{\bp}{\lra} \mathfrak E^{0,\bullet+1}\lra 0.
$$
Here $\mathfrak A^{0,\bullet}$ is  the sheaf   of smooth $(0,\bullet)$-forms and $\mathfrak E^{0,\bullet}$ is  the subsheaf of $\bp$-closed $(0,\bullet)$-forms. 
%all of which represent cohomology of the bicomplex $C^{\bullet}(\U,A^{0,\bullet})$.

The isomorphisms above can be constructed explicitly by using a ``partition of unity" $1=\sum_i  \frac{z_i\bar z_i}{|\zz|^2}$. %in \cite{Griffiths-Harris}. 
Let 
$$
%\omega_p= \left ( \frac{(n-1-p)!}{(2\pi i)^{n}}|\zz|^{-2(n-p)}\frac {(-1)^{\frac {n(n-1)}{2}}} {z_{i_1}\cdots z_{i_{p+1}}}\sum_{\substack{r=1,\cdots,{p+1} \\ k\neq i_1,\cdots,\widehat{i_r},\cdots,i_{p+1}}}  \bar z_{k}z_{i_r} \iota _{\bar \partial_{k}}d\bar z_{i_r} \iota _{\bar \partial_{i_1}}\cdots\iota _{\bar \partial_{i_{p+1}}}d^{n}\bar z \right )_{i_1\cdots i_{p+1}},
\omega_p= \left ( \frac{(n-1-p)!}{(2\pi i)^{n}}|\zz|^{-2(n-p)}\frac {(-1)^{\frac {n(n-1)}{2}}} {z_{i_1}\cdots z_{i_{p+1}}}\sum_{r=1,\cdots,{p+1}}\sum_{k\neq i_1,\cdots,\widehat{i_r},\cdots,i_{p+1}}  \bar z_{k}z_{i_r} \iota _{\bar \partial_{k}}d\bar z_{i_r} \iota _{\bar \partial_{i_1}}\cdots\iota _{\bar \partial_{i_{p+1}}}d^{n}\bar z \right )_{i_1\cdots i_{p+1}},
$$
$$
\alpha_p=\left (\frac{(n-2-p)!}{(2\pi i)^{n}}|\zz|^{-2(n-1-p)}\frac {(-1)^{\frac {n(n-1)}{2}}} {z_{i_1}\cdots z_{i_{p+1}}}\sum_{k\neq i_1,\cdots,i_{p+1}}  \bar z_{k} \iota _{\bar \partial_{k}} \iota _{\bar \partial_{i_1}}\cdots\iota _{\bar \partial_{i_{p+1}}}d^{n}\bar z \right )_{i_1\cdots i_{p+1}}.
$$
Then 
$\omega_0=\omega$, $\omega_{n-1}=\gamma$, and
$\bp\alpha_p=-\omega_p$, $\delta\alpha_p=\omega_{p+1}$.

%Indeed, using a ``partition of unity" $1=\sum_i  \frac{z_i\bar z_i}{|\zz|^2}$, we can solve $\bp\alpha_p=-\omega_p$, $\delta\alpha_p=\omega_{p+1}$ successively (here $\omega_0=\omega$, $\omega_{n-1}=\gamma$)

Under the above chain of isomorphisms, $\gamma\in H^{n-1}(\U,\Hol)$ corresponds to $\omega\in H^{0,n-1}_{\bar\partial}(\DD^n-0)$. 
%which is both $\bp$ and $\delta$-closed. ... 
Moreover, these isomorphisms do not affect $\partial_i$. Thus $\overline{\C[\partial_1,\cdots,\partial_n]\omega}$ is a space of representatives for $H^{0,n-1}_{\bar\partial}(\DD^n-0)$. 
\end{proof}

The following lemma is a slight generalization of Proposition \ref{Dolbeault-punctured-disk} and will be used later. For any subset $V\subset \C^k\times \C^n$, denote $V^*=V\cap (\C^k\times \lr{0})$ and $V^\circ=V\cap (\C^k\times (\C^n-0))$.  Let $p_1:\C^k\times \C^n \to \C^k$ be the projection onto the first factor. 
\begin{lem}\label{open-puncture}
Suppose $V\subset \C^k\times \C^n$ is a convex open set. Then 
$H^{0,\bullet}_\bp(V^\circ)=\begin{cases}
\Hol(V) & \bullet=0 \\
\Hol(p_1(V^*)) \otimes  \overline{\C[\partial_1,\cdots,\partial_n]\omega} & \bullet =n-1 \\
0 & \text{else}
\end{cases}$.
\end{lem}
\begin{proof}
Let  $\mathcal V =\lr{V_i}_{i\in I}$ be 
a polydisk open cover of $V$. It induces a polydisk open cover $p_1(\mathcal V^*):=\lr{p_1(V_i^*)}_{i\in I, V_i^* \ne\emptyset}$ of $p_1(V^*)$, and an open cover $\mathcal V^\circ:=\lr{V_i^\circ}_{i\in I}$ of $V^\circ$. Notice that the Dolbeault cohomology of $V_i^\circ$ is simple: $V_i^\circ$ is of the form $\DD^k\times \DD^n$ or $\DD^k\times(\DD^n-0)$, and it is easy to see
$$
H^{0,\bullet>0}_\bp(\DD^k\times(\DD^n-0))= \Hol(\DD^k) \otimes \overline{\C[\partial_1,\cdots,\partial_n]\omega}.
$$
using \v{C}ech-Dolbeault cohomology with respect to an cover $\{\DD^k\times \lr{z_i\ne 0}\}$ of $\DD^k\times(\DD^n-0)$. 
 
Now we can compute cohomology of the \v{C}ech-Dolbeault bicomplex $C^{\bullet}(\mathcal V^\circ,A^{0,\bullet})$.  
%The $E_1$ page $H^{\bullet}_\bp C^{\bullet}(\mathcal V,A^{0,\bullet})$ is of the form
The spectral sequence associated with the bicomplex (with cohomological Serre grading) collapses at %the $E_2$ page %$H^{\bullet}_\delta H^{\bullet}_\bp C^{\bullet}(\mathcal V,A^{0,\bullet})$ 
page $2$

\begin{center}
\begin{sseqpage}[cohomological Serre grading, no x ticks, no y ticks, x range = {0}{2}, y range = {0}{2},   classes = {draw = none },  yscale = 0.9  
]
\begin{scope}[ background ]
\node at (0.6,-2) {E_1 \text{ Page} \quad H^{\bullet}_\bp C^{\bullet}(\mathcal V^\circ,A^{0,\bullet}) };

\node at (0, -1) {0};
\node at (1, -1) {1}; 
\node at (2, -1) {\cdots}; 

\node at (-1, 0) {0}; 
\node at (-1, 1) {\vdots}; 
\node at (-1, 2) {n-1}; 
\end{scope}

% elements
\class["\,\#\,"] (0,0) 
\class["\,\#\,"](1,0) 
\class["\,\cdots\,"](2,0)

\class["\,*\,"](0,2)  
\class["\,*\,"](1,2) 
\class["\,\cdots\,"](2,2)

% arrows  
\d["\delta"']1 (0,0)
\d["\delta"']1 (1,0) 

\d["\delta"']1 (0,2)
\d["\delta"']1 (1,2) 
\end{sseqpage}  
\qquad\qquad\qquad
\begin{sseqpage}[cohomological Serre grading, no x ticks, no y ticks, x range = {0}{2}, y range = {0}{2},   classes = {draw = none },  yscale = 0.9
]
\begin{scope}[ background ]
\node at (0.65,-2) {E_2 \text{ Page} \quad H^{\bullet}_\delta H^{\bullet}_\bp C^{\bullet}(\mathcal V^\circ,A^{0,\bullet}) };

\node at (0, -1) {0};
\node at (1, -1) {1}; 
\node at (2, -1) {\cdots}; 

\node at (-1, 0) {0}; 
\node at (-1, 1) {\vdots}; 
\node at (-1, 2) {n-1}; 
\end{scope}

% elements
\class["\,\blacksquare\,"](0,0)  
\class["\,\square \,"] (0,2)  
 
\end{sseqpage}    
\end{center}
Here %we have used 
\begin{itemize}
\item $*\overset{\delta}{\rightarrow} *\overset{\delta}{\rightarrow} \cdots $ is 
the complex 
%isomorphic to %$(C^\bullet(p_1(\mathcal V^*),\Hol),\delta) \otimes \Id_{\overline{\C[\partial_1,\cdots,\partial_n]\omega}}$. 
$(C^\bullet(p_1(\mathcal V^*),\Hol) \otimes \overline{\C[\partial_1,\cdots,\partial_n]\omega},\delta \otimes \Id)$ 
with cohomology  
$$
\square=\Hol(p_1(V^*))  \otimes  \overline{\C[\partial_1,\cdots,\partial_n]\omega}.
$$
%which can be seen by using a \v Cech resolution.% associated with the cover $\lr{z_i\ne 0}$ of $\C^n-0$.
%  changing $\overline{\C[\partial_1,\cdots,\partial_n]\omega} \to \overline{\C[\partial_1,\cdots,\partial_n]\gamma}$, and the latter is just the space of holomorphic functions. 
\item $\#\overset{\delta}{\rightarrow} \#\overset{\delta}{\rightarrow} \cdots $ coincides with the \v Cech complex %$C^\bullet(\mathcal V^\circ,\Hol)$, which coincides with 
$(C^\bullet(\mathcal V,\Hol),\delta)$ by Hartogs' extension theorem. \\ Its cohomology is $\blacksquare =\Hol(V)$.
\end{itemize}
%We see the spectral sequence %of the bicomplex  collapses at the $E_2$ page $H^{\bullet}_\delta H^{\bullet}_\bp C^{\bullet}(\mathcal V^\circ,A^{0,\bullet})$ 
%where $\square=\Hol(p(V^*))  \otimes  \overline{\C[\partial_1,\cdots,\partial_n]\omega}$  and $\blacksquare =\Hol(V)$.

\end{proof}

\subsection{Residues} 
Let $\Omega^{0,n-1}(\DD^n-0)$ be the space of holomorphic $(0,n-1)$-forms on punctured polydisk. 
In this section, we find a retraction of the embedding $\overline{\C[\partial_1,\cdots,\partial_n]\omega}\hookrightarrow \Omega^{0,n-1}(\DD^n-0)$, which kills all $\bp$-exact forms. We first need two lemmas.

Let $S^{2n-1}$ be any sphere in $\DD^n$ with center $0$. Since $\omega_{BM}(\zz)$ is  the integral kernel for the inverse of $\bp$, we have
\begin{lem}
Let  $i_1,\cdots,i_n\ge0$. Then
$$\int_{S^{2n-1}} \,dz^1\cdots dz^n\, z_1^{i_1}\cdots z_n^{i_n} \omega=
\begin{cases}
1 & i_1=\cdots=i_n=0 \\
0 & \text{else}
\end{cases}.
$$
\end{lem}

\begin{lem} \label{vanish-lie-derivative}
Let $\alpha$ be a holomorphic $(n,n-1)$-form on $\DD^n-0$ and $X$ be a vector field. Then $\int_{S^{2n-1}} \mathcal L_{X}\alpha=0$.
\end{lem}
\begin{proof}
Since $d\alpha=\bp\alpha=0$, we have $\mathcal L_{X}\alpha=d\iota_X\alpha+\iota_X d\alpha=d\iota_X\alpha$.  So
$$
\int_{S^{2n-1}} \mathcal L_{X}\alpha 
=\int_{S^{2n-1}} d\iota_X\alpha=0.
$$
\end{proof}

\begin{defn}
We define the residue map on the space $\Omega^{0,n-1}(\DD^n-0)$ of holomorphic $(0,n-1)$-forms on $\DD^n-0$
\ali{
\oint: \Omega^{0,n-1}(\DD^n-0) &\lra \C[[\partial_1,\cdots,\partial_n]]\omega \\
  \beta       &\lmt \sum_{i_1,\cdots,i_n {\ge 0}}\LR{\int_{S^{2n-1}} \,dz^1\cdots dz^n\, \frac{(-z_1)^{i_1}}{i_1!}\cdots \frac{(-z_n)^{i_n}}{i_n!} \beta} \partial_{1}^{i_1}\cdots\partial_{n}^{i_n}\omega
}
\end{defn}
Since $\bp\beta=0$, the result is independent of the choice of $S^{2n-1}$.  

\begin{lem}
When restricted to the subspace $\overline{\C[\partial_1,\cdots,\partial_n]\omega} \subset \Omega^{0,n-1}(\DD^n-0)$, the residue $\oint$ coincides with the natural inclusion $\overline{\C[\partial_1,\cdots,\partial_n]\omega} \hookrightarrow  \C[[\partial_1,\cdots,\partial_n]]\omega$. 
\end{lem}
\begin{proof}
Taking $X=\partial_i$ in Lemma \ref{vanish-lie-derivative}, we see that $\oint \partial_{1}^{i_1}\cdots\partial_{n}^{i_n}\omega = \partial_{1}^{i_1}\cdots\partial_{n}^{i_n}\omega$ by an integration-by-parts argument.
\end{proof}
 
\begin{prop}
$\ker\oint%=\Im  \bar\partial_{(n-2)}
$ is the space of $\bp$-exact $(0,n-1)$-forms, and $\Im \oint=\overline{\C[\partial_1,\cdots,\partial_n]\omega}$.
\end{prop}
\begin{proof}
Since $\oint\bp(-)=\oint d(-)=0$, the image of $\bar\partial_{(n-2)}: A^{0,n-2}(\DD^n-0)\to \Omega^{0,n-1}(\DD^n-0)$ lies in $\ker\oint$ 
%Since $\bar\partial=\sum_i d\bar z_i \mathcal L_{\bar\partial_i}$, $\oint\bp(-)=0$ by Lemma \ref{vanish-lie-derivative}.
%So there is an inclusion $\Im \bar\partial_{(n-2)} \to \Ker \oint$ which makes the following   diagram  commutative
$$
\begin{tikzcd}
0 \arrow[r] & \Im  \bar\partial_{(n-2)} \arrow[r] \arrow[d, hook] & \Omega^{0,n-1}(\DD^n-0) \arrow[d, phantom, sloped, "=\!=\!\!="]   &   &   \\
0 \arrow[r] & \Ker\oint \arrow[r] & \Omega^{0,n-1}(\DD^n-0) \arrow[r, "\oint"] & \C[[\partial_1,\cdots,\partial_n]]\omega   &  
\end{tikzcd}
$$
By the universal property of cokernels, the diagram extends  to
$$
\begin{tikzcd}
0 \arrow[r] & \Im  \bar\partial_{(n-2)} \arrow[r] \arrow[d] & \Omega^{0,n-1}(\DD^n-0) \arrow[d, phantom, sloped, "=\!=\!\!="] \arrow[r] & H^{0,n-1}_{\bar\partial}(\DD^n-0) \arrow[r] \arrow[d] & 0 \\
0 \arrow[r] & \Ker\oint \arrow[r] & \Omega^{0,n-1}(\DD^n-0) \arrow[r, "\oint"] & \C[[\partial_1,\cdots,\partial_n]]\omega   &  
\end{tikzcd}
$$
\iffalse
$$
\begin{tikzcd}
0 \arrow[r] & \Im  \bar\partial_{(n-2)} \arrow[r] \arrow[d, dashed] & \Omega^{0,n-1}(\DD^n-0) \arrow[d, phantom, sloped, "=\!=\!\!="] \arrow[r] & H^{0,n-1}_{\bar\partial}(\DD^n-0) \arrow[r] \arrow[d, hook] & 0 \\
0 \arrow[r] & \Ker\oint \arrow[r] & \Omega^{0,n-1}(\DD^n-0) \arrow[r, "\oint"] & \C[[\partial_1,\cdots,\partial_n]]\omega   &  
\end{tikzcd}
$$
\fi
Since $H^{0,n-1}_{\bar\partial}(\DD^n-0)$ has a set of representatives $\overline{\C[\partial_1,\cdots,\partial_n]\omega}$, the map $H^{0,n-1}_{\bar\partial}(\DD^n-0) \to \C[[\partial_1,\cdots,\partial_n]]\omega$ is injective with image $\overline{\C[\partial_1,\cdots,\partial_n]\omega}$. 
By the five lemma, the map $\Im \bar\partial_{(n-2)} \to \Ker \oint$ is an isomorphism.
\end{proof}
\iffalse
As a result, $\oint \beta$ only depends on the cohomology class of $\beta$. 
%Since $\oint$ is the identity map when restricted to $\overline{\C[\partial_1,\cdots,\partial_n]\omega}$, we have
Moreover, the above diagram shows the image of $\oint$ coincides with the image of the embedding $H^{0,n-1}_{\bar\partial}(\DD^n-0) \hookrightarrow \C[[\partial_1,\cdots,\partial_n]]\omega$, so %Since $H^{0,n-1}_{\bar\partial}$ has a set of representatives of convergent series, 
we have
\fi
\begin{cor}\label{residue-convergence}
$\oint \beta\in \C[[\partial_1,\cdots,\partial_n]]\omega$ converges on $\C^n-0$ and  depends only on the cohomology class of $\beta$. 
\end{cor}

\subsection{Dolbeault cohomology group of $\Conf_m(\C^n)$}
In this section, we apply the Leray spectral sequence to the map
\ali{
\pi:\qquad \Conf_m(\C^n)         &\lra \Conf_{m-1}(\C^n)\\ 
  (\zz^1,\cdots,\zz^{m-1},\zz^m) &\lmt   (\zz^1,\cdots,\zz^{m-1}) 
}
to compute cohomology of the structure sheaf $\Hol_{\Conf_m(\C^n)}$ of $\Conf_m(\C^n)$. Note that a convex domain in $\C^n$ is Stein (see for example \cite{Stein}); % Chapter IV, page 110 
therefore, it suffices to use a convex cover. 

Let us analyze the stalks of $\pi$. Let $p=(\zz^1,\cdots,\zz^{m-1})$ be a point in $\Conf_{m-1}(\C^n)$, then $\pi\inv(p)=\C^n-\lr{\zz^1,\cdots,\zz^{m-1}}$, and its higher Dolbeault cohomology $H^{0,\bullet>0}_{\bar\partial}(\pi\inv(p))$ is generated by holomorphic derivatives of $\omega_{im}:=\omega(\zz^i-\zz^m,\bar\zz^i-\bar\zz^m)$, $1\le i\le m-1$. When we consider a convex neighborhood $U$ of $p$ instead of $p$, the inverse image $\pi\inv(U)$ becomes $U\times \C^n$ with $m-1$ hypersurfaces $\lr{\zz^i=\zz^m}$ removed. Next, we show $H^{0,\bullet>0}_{\bar\partial}(\pi\inv(U))$ still arises from the forms $\omega_{im}$. 
For simplicity, assume $U$ is of the form $V_1\times \cdots \times V_{m-1}$, where $V_i$ are convex open sets with pairwise disjoint closures. Let us find an open cover of $\pi\inv(U)$. 
For $1\le i\le m-1$, let $W_i$ be convex open sets in $\C^n$ with disjoint closure, such that $\overline V _i\subset W_i$. Choose a convex open cover $\lr{W_{m},W_{m+1},\cdots}$ of $\C^n-\cup \overline V_i$, so we have a convex open cover $\lr{W_i}_{i\ge 1}$ of $\C^n$. Let $p: \pi\inv(U) \to \C^n$ be the projection $p(\zz^1,\cdots,\zz^m) = \zz^m$ and $X_i=p\inv(W_i)$. $\mathcal X=\lr{X_i}_{i\ge 1}$ is an open cover of $\pi\inv(U)$. For $1\le i\le m-1$, 
$X_{i}=\lr{(\zz^i,\zz^m)|\zz^i\in V_i,\zz^m\in W_i,\zz^m\ne \zz^i}  \times V_1\times\cdots\times \widehat V_i\times\cdots\times V_{m-1}$.  Change coordinates   $(\zz^i,\zz^m)\to (\zz^i,\zz^i-\zz^m)$,   Lemma \ref{open-puncture} says  $H^{0,\bullet_{>0}}_{\bp}(X_i)=\Hol(U) \otimes \overline{\C[\partial_{z^m_1},\cdots,\partial_{z^m_n}]\omega_{im}}$.  
%(Here we have used $( {\partial_{z^i_j}}+ {\partial_{z^m_j}})\omega_{im}=0$ to eliminate $\vec \partial_i$.)
Moreover, $X_{i}$ for $i\ge m$, or any nontrivial intersection of $X_i$, is convex, so its $\bp$-cohomology vanishes.  
Then we can compute the cohomology of 
the \v{C}ech-Dolbeault bicomplex %associated with the open cover $\lr{X_i}$ of $\pi\inv(U)$, we see 
$C^{\bullet}(\mathcal X,A^{0,\bullet})$. The spectral sequence  
collapses at page $2$
\begin{center}
\begin{sseqpage}[cohomological Serre grading, no x ticks, no y ticks, x range = {0}{2}, y range = {0}{2},   classes = {draw = none },  yscale = 0.9  
]
\begin{scope}[ background ]
\node at (0.6,-2) {E_1 \text{ Page} \quad H^{\bullet}_\bp C^{\bullet}(\mathcal X,A^{0,\bullet}) };

\node at (0, -1) {0};
\node at (1, -1) {1}; 
\node at (2, -1) {\cdots}; 

\node at (-1, 0) {0}; 
\node at (-1, 1) {\vdots}; 
\node at (-1, 2) {n-1}; 
\end{scope}

% elements
\class["\,\#\,"] (0,0) 
\class["\,\#\,"](1,0) 
\class["\,\cdots\,"](2,0)

\class["\,\square\,"](0,2)   

% arrows  
\d["\delta"']1 (0,0)
\d["\delta"']1 (1,0) 
 
\end{sseqpage}  
\qquad\qquad\qquad
\begin{sseqpage}[cohomological Serre grading, no x ticks, no y ticks, x range = {0}{2}, y range = {0}{2},   classes = {draw = none },  yscale = 0.9
]
\begin{scope}[ background ]
\node at (0.65,-2) {E_2 \text{ Page} \quad H^{\bullet}_\delta H^{\bullet}_\bp C^{\bullet}(\mathcal X,A^{0,\bullet}) };

\node at (0, -1) {0};
\node at (1, -1) {1}; 
\node at (2, -1) {\cdots}; 

\node at (-1, 0) {0}; 
\node at (-1, 1) {\vdots}; 
\node at (-1, 2) {n-1}; 
\end{scope}

% elements
\class["\,\blacksquare\,"](0,0)  
\class["\,\square \,"] (0,2)  
 
\end{sseqpage}    
\end{center}
Here  $\# \overset{\delta}{\rightarrow} \# \overset{\delta}{\rightarrow}  \cdots $ is the \v Cech complex  
$(C^\bullet(\mathcal X,\Hol),\delta)$ with cohomology $\blacksquare =\Hol(\pi\inv(U))$, and $\square= \Hol(U) \otimes \oplus_{i=1}^{m-1}\overline{\C[\partial_{z^m_1},\cdots,\partial_{z^m_n}]\omega_{im}}.$ 
So 
$$
H^{0,\bullet}_\bp(\pi\inv(U))=
\begin{cases}
\Hol(\pi\inv(U)) & \bullet=0 \\
 \Hol(U) \otimes \oplus_{i=1}^{m-1}\overline{\C[\partial_{z^m_1},\cdots,\partial_{z^m_n}]\omega_{im}} & \bullet=n-1 \\
0 & \text{else}
\end{cases}.
$$
In other words, 
$H^{0,\bullet}_\bp(\pi\inv(U))=\Hol(U) \otimes \LR{\Hol(\C_m)\oplus \oplus_{i=1}^{m-1} \overline{\C[\partial_{z^m_1},\cdots, \partial_{z^m_n}]\omega_{im}}}$, where $\Hol(\C_m)$ denotes the space of holomorphic functions in the variable $\zz^m$. As a result, the push-forward sheaf 
$$
R^\bullet \pi_*\Hol_{\Conf_m(\C^n)}
=\Hol_{\Conf_{m-1}(\C^n)} \otimes \LR{\Hol(\C_m)\oplus \oplus_{i=1}^{m-1} \overline{\C[\partial_{z^m_1},\cdots, \partial_{z^m_n}]\omega_{im}}}
$$
is the tensor product of $\Hol_{\Conf_{m-1}(\C^n)}$ with a constant sheaf, with cohomology 
$$
H^\bullet R^\bullet \pi_*\Hol_{\Conf_m(\C^n)}
=H^{0,\bullet}_\bp(\Conf_{m-1}(\C^n))\otimes \LR{\Hol(\C_m)\oplus \oplus_{i=1}^{m-1} \overline{\C[\partial_{z^m_1},\cdots, \partial_{z^m_n}]\omega_{im}}}.
$$ 
Since the forms along fiber exist globally, the Leray spectral sequence collapses at the differential $d_1$. By induction we have 
\begin{thm}\label{Dolbeault-group}
The Dolbeault cohomology group of $\Conf_m(\C^n)$ is
\ali{
H^{0,\bullet}_\bp(\Conf_m(\C^n))
&=\otimes_{j=1}^n \LR{\Hol(\C_j)\oplus \oplus_{i=1}^{j-1} \overline{\C[\partial_{z^j_1},\cdots, \partial_{z^j_n}]\omega_{ij}}} \\
&=\oplus_{k=0}^{m-1}\oplus_{\substack{1< b_1<\cdots<b_{k}\le m\\ a_1<b_1,\cdots,a_k<b_k}} \LR{\otimes_{j\ne b_i}\Hol(\C_j)} \otimes \LR{\otimes_{i=1}^k\overline{\C[\partial_{z^{b_i}_1},\cdots, \partial_{z^{b_i}_n}]\omega_{a_ib_i}}},
}
where $\Hol(\C_j)$ denotes the space of holomorphic functions in the variable $\zz^j$. In other words, $H^{0,\bullet}_\bp(\Conf_m(\C^n))$ is the subspace of the space of smooth differential forms on $\Conf_m(\C^n)$ generated by a Schauder Basis 
$$\lr{f(\zz^1,\cdots,\widehat {\zz ^{b_1}},\cdots,\widehat {\zz ^{b_k}},\cdots,\zz^m)g_1(\vec \partial_{b_1})\omega_{a_1b_1}\cdots g_k(\vec \partial_{b_k})\omega_{a_kb_k}\,|\,f,g_i \text{ are monomials, } 1< b_1<\cdots<b_{k}\le m,a_i<b_i}.$$
\end{thm}

\iffalse
\begin{rem}
Since $\omega_{ij}$ only relies on $\zz^i-\zz^j$, we have $\partial_{z^i_a}\omega_{ij}=-\partial_{z^j_a}\omega_{ij}$, so we can take derivatives only with respect to the second component.
\end{rem}
\fi

\subsection{Ring structure}% on Dolbeault cohomology group of $\Conf_m(\C^n)$}
In this section we describe the ring structure of $H^{0,\bullet}_\bp(\Conf_m(\C^n))$ by generators and relations. 
\begin{prop}[The Arnold relation]\label{Arnold-relation}
%$\omega_{ab}\omega_{bc}+\omega_{bc}\omega_{ca}+\omega_{ca}\omega_{ab}$
$\omega_{12}\omega_{23}+\omega_{23}\omega_{31}+\omega_{31}\omega_{12}$ is $\bp$-exact.
\end{prop}

\begin{proof}
Let $\pi$ be the projection 
\ali{
\pi:\Conf_4(\C^n) &\lra \Conf_3(\C^n) \\
(\zz^1,\zz^2,\zz^3,\zz^4) & \lmt (\zz^1,\zz^2,\zz^3)
}
and denote the fiber integration  by $\int_{\zz^4}$. Then
\ali{
&\bp \int_{\zz^4}  dz_1^4\cdots dz_n^4 \,  \omega_{14}\omega_{24}\omega_{34}
=\int_{\zz^4} dz_1^4\cdots dz_n^4 \,  (\bp_1+\bp_2+\bp_3) \LR{\omega_{14}\omega_{24}\omega_{34}}  
=\omega_{12}\omega_{23}+\omega_{23}\omega_{31}+\omega_{31}\omega_{12}.
}
\end{proof}

Let  
$
\lR{\Hol(\C^n),\,\, \overline{\C[ \partial_{z^j_1},\cdots, \partial_{z^j_n}]\omega_{ij}}} 
$ 
be the topological ring generated by $\Hol(\C^n)$ and $\overline{\C[ \partial_{z^j_1},\cdots, \partial_{z^j_n}]\omega_{ij}}$ for all $i<j$.  
There is a natural map of graded commutative topological rings  
$$
\lR{\Hol(\C^n),\,\, \overline{\C[ \partial_{z^j_1},\cdots, \partial_{z^j_n}]\omega_{ij}}} 
\lra H^{0,\bullet}_\bp(\Conf_m(\C^n)). 
$$  
%Notice that  for holomorphic functions $f,g$, by type reason we have $f(\vec \partial_b)\omega_{ab}  g(\vec \partial_b)\omega_{ab}=0$. 
By Theorem \ref{Dolbeault-group}, the map is surjective, and  its kernel,  as a topological subgroup, is generated by  holomorphic derivatives of $\omega_{ab}\omega_{bc}+\omega_{bc}\omega_{ca}+\omega_{ca}\omega_{ab}$ %, $f(\vec \partial_b)\omega_{ab}  g(\vec \partial_b)\omega_{ab}$, 
and $(f(\zz^a)-f(\zz^b))g(\vec \partial_b)\omega_{ab}$. 
Using residues on $\C^n-0$, we see that $(f(\zz^a)-f(\zz^b))g(\vec \partial_b)\omega_{ab}$ is $\bp$-exact. So the kernel contains only $\bp$-exact forms. So we have

\begin{thm}\label{Dol-ring-structure-v2}
The natural map of graded commutative topological  rings  
$$
\lR{\Hol(\C^n),\,\, \overline{\C[ \partial_{z^j_1},\cdots, \partial_{z^j_n}]\omega_{ij}}} 
\lra H^{0,\bullet}_\bp(\Conf_m(\C^n))
$$
is surjective with kernel generated by holomorphic derivatives of $\omega_{ab}\omega_{bc}+\omega_{bc}\omega_{ca}+\omega_{ca}\omega_{ab}$ %,  $f(\vec \partial_b)\omega_{ab}  g(\vec \partial_b)\omega_{ab}$ 
and $(f(\zz^a)-f(\zz^b))g(\vec \partial_b)\omega_{ab}$, where $f,g$ are holomorphic functions.
When modulo kernel, we get an isomorphism of topological rings.

In other words, $H^{0,\bullet}_\bp(\Conf_m(\C^n))$ is generated by holomorphic functions and holomorphic derivatives of Bochner kernels, subject to the following relations
\begin{itemize}
\item $\omega_{ab}=(-1)^n \omega_{ba}$
\item $f(\vec \partial_b)\omega_{ab}  g(\vec \partial_b)\omega_{ab}=0$, for $n\ge2$
\item (holomorphic derivatives of) $\omega_{ab}\omega_{bc}+\omega_{bc}\omega_{ca}+\omega_{ca}\omega_{ab}=0$
\item $(f(\zz^a)-f(\zz^b))g(\vec \partial_b)\omega_{ab}=0$
\end{itemize} 
\end{thm}

For $n=1$,  
$\omega_{ab}=\frac{1}{z^a-z^b}$ has bidegree $(0,0)$, so we have   
\begin{cor}\label{Dol-ring-structure-C}
$H^{0,\bullet}_\bp (\Conf_m(\C)) 
=H^{0,0}_\bp (\Conf_m(\C)) 
= \Hol(\Conf_m(\C))$. % is generated by $z_i$ and $\frac{1}{z_i-z_j}$ as a topological ring.
\end{cor}

\iffalse
\begin{cor}
There is  an isomorphism  of graded commutative rings
$$
H^{\bullet,\bullet}_{\bp}(\Conf_m(\C^n)) 
\simeq \C[dz^i_j] \otimes H^{0,\bullet}_{\bp}(\Conf_m(\C^n)) 
$$
where $\C[dz^i_j]$ is the graded commutative ring generated by $dz^1_1,\cdots,dz^1_n,\cdots,dz^m_1,\cdots,dz^m_n$.
\end{cor}
\fi

Since the holomorphic cotangent bundle of $\Conf_m(\C^n)$ is trivial, we have
\begin{cor}
There is  an isomorphism  of graded commutative rings
$$
H^{\bullet,\bullet}_{\bp}(\Conf_m(\C^n)) 
\simeq \C[dz^1_1,\cdots,dz^1_n,\cdots,dz^m_1,\cdots,dz^m_n] \otimes H^{0,\bullet}_{\bp}(\Conf_m(\C^n)).
$$ 
\end{cor}

\section{A spectral sequence for $\Conf_m(X)$}\label{ss-confX}

Let $X$ be a  complex manifold  with dimension $n\ge2$  and $\Conf_m(X)$ be the configuration space of $m$ pairwise different points in $X$. 
In this section, we study the Leray spectral sequence associated to the inclusion $f:\Conf_m(X)\hookrightarrow X^m$ and the structure sheaf $\HolX{\Conf_m(X)}$ on $\Conf_m(X)$ 
$$
H^i(X^m;R^jf_* \HolX{\Conf_m(X)}) \Longrightarrow H^{i+j} (\Conf_m(X); \HolX{\Conf_m(X)}).
$$  
  
Denote by $\mathcal D_X$ the sheaf of differential operators on $X$,  by $K_X$  the canonical bundle of $X$,  by $\Delta:X\to X^2$ the diagonal embedding, and by $p_{ij}:\Conf_m(X)\to \Conf_2(X)$  the $(i,j)$-th projection. Denote the $(i,j)$-th diagonal in $X^m$ by $\Delta_{ij}=p_{ij}^{-1}\Delta(X)$. 
Let $f_2:\Conf_2(X)\hookrightarrow X^2$ be the inclusion and denote $G=R^{n-1}{f_2}_* \HolX{\Conf_2(X)}$. The sheaf $G$ is supported on $\Delta(X)$, and its space of local sections is $\overline{\C[\partial_{z^2_1},\cdots, \partial_{z^2_n}]\omega_{12}}$. 
Recall $\omega_{BM}=dz_1\cdots dz_n \omega$. By Proposition \ref{BM-coor-change}, $\bp$-cohomology classes of $\omega_{BM}$ in different coordinates coincide, so $\omega_{BM}(\zz^1-\zz^2,\bar\zz^1-\bar\zz^2)$ spans a constant sheaf supported on $\Delta(X)$, which is isomorphic to $\Delta_*\C$. Consequentially, the $\bp$-cohomology class of $\omega_{12}$ is a local section of  $\Delta_*\Gamma(X,K_X\inv)$, which can be found globally if $X$ admits a holomorphic volume form. 
In local coordinates on $X$, the sheaf of higher direct images $R^jf_* \HolX{\Conf_m(X)}$ is identified as that on $\Conf_m(\C^n)$, which is by Theorem  \ref{Dolbeault-group} generated by $\HolX{X^m}$ and the $G_{ij}:=p_{ij}^*G$.  
%Recall $\omega_{BM}=dz_1\cdots dz_n \omega$. By Proposition \ref{BM-coor-change}, $\bp$-cohomology classes of $\omega_{BM}$ in different coordinates coincide, so $\omega_{BM}$ spans a constant sheaf supported on $\Delta_{ij}$, which is isomorphic to $p_{ij}^*\Delta_*\C$. Consequentially, the $\bp$-cohomology class of $\omega_{ij}$ is a local section of  $p_{ij}^*\Delta_*\Gamma(X,K_X\inv)$, which can be found globally if $X$ admits a holomorphic volume form. 
 
By Theorem \ref{Dolbeault-group},  as a sheaf of topological groups,
$$
R^\bullet f_* \HolX{\Conf_m(X)}
=\oplus_{k=0}^{m-1}\oplus_{\substack{1< b_1<\cdots<b_{k}\le m\\ a_1<b_1,\cdots,a_k<b_k}} 
G_{a_1b_1} \otimes_{\HolX{X^m}} \cdots \otimes_{\HolX{X^m}} G_{a_kb_k}
%{\otimes_{\HolX{X^m}}}_{1\le i\le k}  G_{a_ib_i}
%p_{a_ib_i}^*\Delta_* \Gamma \LR{X, \mathcal D_X \otimes  K_X\inv}.
$$
where $G_{a_1b_1} \otimes_{\HolX{X^m}} \cdots \otimes_{\HolX{X^m}} G_{a_kb_k}$ %${\otimes_{\HolX{X^m}}}_{1\le i\le k} G_{a_ib_i}$ 
is supported on the diagonal $\cap_i\Delta_{a_ib_i}$. % with cohomology  
%$$
%H^\bullet \LR{X^m, G_{a_1b_1} \otimes_{\HolX{X^m}} \cdots \otimes_{\HolX{X^m}} G_{a_kb_k}} =?
%$$
\iffalse
$$
H^\bullet \LR{X^m, {\otimes_{\HolX{X^m}}}_{1\le i\le k}  p_{a_ib_i}^*\Delta_* \Gamma \LR{X, \mathcal D_X \otimes  K_X\inv}}
%=\Hol(\cap_{i=1}^k \Delta_{a_ib_i}) \otimes H_{a_1b_1}\otimes\cdots\otimes H_{a_kb_k}.
=H_{a_1b_1}\otimes_{\Hol(X^{m})}\cdots\otimes_{\Hol(X^{m})} H_{a_kb_k}.
$$
\fi
% (since it is  tensor product of locally free sheaves  on the diagonal)
So as a topological group, 
$$
H^\bullet \LR{X^m, R^\bullet f_* \HolX{\Conf_m(X)}}
%=\oplus_{k=0}^{m-1}\oplus_{\substack{1< b_1<\cdots<b_{k}\le m\\ a_1<b_1,\cdots,a_k<b_k}} \Hol(\cap_{i=1}^k \Delta_{a_ib_i}) \otimes H_{a_1b_1}\otimes\cdots\otimes H_{a_kb_k}.
=H^\bullet \LR{X^m,
\oplus_{k=0}^{m-1}\oplus_{\substack{1< b_1<\cdots<b_{k}\le m\\ a_1<b_1,\cdots,a_k<b_k}} 
G_{a_1b_1} \otimes_{\HolX{X^m}} \cdots \otimes_{\HolX{X^m}} G_{a_kb_k}}.
$$
The product structure  on $H^\bullet (X^m,R^\bullet f_* \HolX{\Conf_m(X)})$ is induced by the product on $R^\bullet f_* \HolX{\Conf_m(X)}$ 
%(using $H_{ij}^2=0$ and the Arnold relation and the diagonal $p_i^*f-p_j^*f$) 
and the cup product. 
Let us analyze the differential $d_n$. 
% By functoriality of Leray spectral sequences, 
Since $d_n$ is compatible with the ring structure, 
% since the spectral sequence is a functor of rings
it is enough to determine by the image of the diagonal classes in $H^i(X^m;R^{n-1}f_* \HolX{\Conf_m(X)})$. 
This is quite similar to the Gysin map for sphere bundles, but here the fiber is $\C^n-0$ rather than a sphere, and the fiber integration is replaced by the $\bp$-integration \cite{Suwa-Cech}. 
The locally defined Bochner kernels in $H^{n,n-1}_\bp(U)$ need not glue to a globally defined cohomology class. The Gysin map sends them to the top Atiyah class \cite{Suwa-Cech} of the normal bundle, which is of type $(n,n)$ and measures the failure of gluing. $d_n$ is the Gysin map twisted by inverse of the canonical bundle of the diagonal.

Finally, we note that 
% there is a natural isomorphism $f_*\Omega^{\bullet,0}_{\Conf_m(X)} \cong \Omega^{\bullet,0}_{X^m} \otimes_{\HolX{X^m}} f_*\HolX{\Conf_m(X)}$
we can also consider $(\bullet,\bullet)$-forms using 
\ali{
R^\bullet f_*\Omega^{\bullet,0}_{\Conf_m(X)}
\cong R^\bullet f_*f^*\Omega^{\bullet,0}_{X^m}
\cong \Omega^{\bullet,0}_{X^m} \otimes_{\HolX{X^m}} R^\bullet f_*\HolX{\Conf_m(X)}.
}
\iffalse
There is a  natural map
$$
H^\bullet \LR{X^m, \lR{\Omega^{\bullet,0}_{X^m},\, G_{ij}}_{\otimes_{\HolX{X^m}}}}
\lra H^\bullet \LR{X^m, R^\bullet f_* \HolX{\Conf_m(X)}}
$$
which identifies the latter with the quotient ring of the former,   modulo the following relations
\begin{itemize}
\item 
$\sigma_{ij}\sigma'_{ij} 
\begin{cases}
=0 &    n\ge2 \\
\in \Omega^{0,\bullet} \LR{X^m, \Omega^{\bullet,0}_{X^m}\otimes_{\HolX{X^m}} G_{ij}} & n=1
\end{cases} $
\item $\sigma_{ij}\sigma_{jk}+\sigma_{jk}\sigma_{ki}+\sigma_{ki}\sigma_{ij}=0$ 
\end{itemize}
Here $\sigma_{ij}, \sigma'_{ij}$ are local sections of  $\Omega^{0,\bullet} \LR{X^m, \Omega^{\bullet,0}_{X^m}\otimes_{\HolX{X^m}} G_{ij}}$  and the other two are similar.
\fi 
 
\begin{thm}\label{conf-mfd-cohomology}
Let $X$ be a complex manifold with dimension $n\ge2$. The inclusion $\Conf_m(X)\hookrightarrow X^m$ determines a Leray spectral sequence which converges to $H^{\bullet,\bullet}_\bp(\Conf_m(X))$ as a graded commutative topological ring. The $E_2$ term is a sheaf of rings with underlying group
$$
H^\bullet \LR{X^m, 
\oplus_{k=0}^{m-1}\oplus_{\substack{1< b_1<\cdots<b_{k}\le m\\ a_1<b_1,\cdots,a_k<b_k}} 
\Omega^{\bullet,0}_{X^m} \otimes_{\HolX{X^m}} G_{a_1b_1} \otimes_{\HolX{X^m}} \cdots \otimes_{\HolX{X^m}} G_{a_kb_k}}.
$$   
The differential $d_n$ is given by the Gysin map, as described above. 
\end{thm}

%When $\dim X=1$, diagonal classes have degree $0$, so we need to consider the ring generated by holomorphic functions and diagonal classes. The result is similar. 
% Notice that for a genus $g$ compact Riemann surface, $H^{0,0}=\C$, $H^{1,0}=\C^g$ holomorphic 1-forms, $H^{0,1}=\C^g$, $H^{1,1}=\C$ volume form. 
\iffalse
\begin{prop}
Let $X$ be a Riemann surface. Then $H^{\bullet,\bullet}_{\bp}(\Conf_m(X))=\Omega^{\bullet,0}(\Conf_m(X))$ is the space of holomorphic $(\bullet,0)$ forms on $\Conf_m(X)$.
\end{prop}
\fi

\subsection{Example: complex tori}    
Let $\Gamma\subset \C^n$ be a lattice of rank $2n$.  Then $X=\C^n/\Gamma$ is a  complex torus. In this section, we compute $H^{\bullet,\bullet}_\bp(\Conf_m(\C^n/\Gamma))$. The $E_2$ page of the spectral sequence is
$$
\oplus_{k=0}^{m-1}\oplus_{\substack{1< b_1<\cdots<b_{k}\le m\\ a_1<b_1,\cdots,a_k<b_k}} 
H^\bullet \LR{X^m, \Omega^{\bullet,0}_{X^m} \otimes_{\HolX{X^m}} G_{a_1b_1} \otimes_{\HolX{X^m}} \cdots \otimes_{\HolX{X^m}} G_{a_kb_k}}.
$$  
Notice that all these sheaves come from trivial bundles on $X^m$ or diagonals, and the spectral sequence is also valid for $n=1$. % since we can choose  coordinate system $z\to w=z+a$, then $dw/w\ne dz/z$, but we can lift them to $\C$, where the cohomology of $1/z$ spans a trivial bundle, invariant under the action of $\Gamma$.
%To take higher differentials $d_n$ is equivalently to find 
%subquotients, % of forms that glue to a global class, module exact forms.  
Let us determine which cohomology class  can be glued together to form a global one. We perform a coordinate change  
$$
(\zz^1,\cdots,\zz^m)\to (\zz^1,\cdots,\widehat {\zz ^{b_1}},\cdots,\widehat {\zz ^{b_k}},\cdots,\zz^m,\zz^{b_1}-\zz^{a_1},\cdots,\zz^{b_k}-\zz^{a_k}),
$$
then each term in $E_2$ page is a sum and each summand is of the form
$$
f_{1}(\partial_{a_1})\omega(\zz^{a_1},\bar\zz^{a_1})\cdots  f_{k}(\partial_{a_k})\omega(\zz^{a_k},\bar\zz^{a_k}) 
$$
up to multiplication by some $dz^i_j,d\bar z^i_j$. Assume that this form represents a global cohomology class. Taking the residue for $\zz^{a_1},\cdots, \widehat{\zz^{a_{l}}},\cdots,\zz^{a_k}$, we see $f_{l}(\partial_{a_l})\omega(\zz^{a_l},\bar\zz^{a_l})$ is a holomorphic form on $\C^n/\Gamma-0$ for every $l$. Using \v Cech cohomology we can find the generator of such forms.  
Let 
$\omega'(\zz,\bar \zz)= %(-1)^n  
\partial_{1}\cdots\partial_{n}\omega(\zz,\bar \zz)$ and
$$
\widetilde\omega(\zz,\bar \zz)=
\omega'(\zz,\bar \zz)+\sum_{\substack{\eta\in \Gamma\\\eta\ne 0}} \LR{\omega'(\zz-\eta,\bar\zz-\bar\eta)+\omega'(\eta,\bar\eta)},
$$
then a holomorphic form on $\C^n/\Gamma-0$ is a linear combination of $\widetilde\omega(\zz,\bar \zz)$ 
and its holomorphic derivatives of all orders. %By lifting along the projection  $\C^n/\Gamma\to \C^n$,
% Using residues we see their product is not $\bp$-exact.
%locally choose an open subset being product of the form $\DD^n-0$ and taking residues, we see these forms are not $\bp$-exact even locally. 

Define the map $q_{ij}$ by
\ali{
q_{ij}:\Conf_m(\C^n/\Gamma)&\lra \C^n/\Gamma-0 \\
       (\zz^1,\cdots,\zz^m)&\lmt \zz^a-\zz^b
} 
and $\widetilde\omega_{ij}=q_{ij}^*\widetilde\omega$. 
Let 
$$
\lR{\C[dz^i_j,d\bar z^i_j], \, \overline{\C[ \partial_{z^j_1},\cdots, \partial_{z^j_n}]\widetilde\omega_{ij}}} 
$$
be the graded commutative topological ring  generated by  constant forms $1, dz^i_j,d\bar z^i_j$  and holomorphic derivatives of $\widetilde\omega_{ij}$ with constant coefficients. Using residues, we find 
\begin{prop}
For $n\ge1$, 
the natural map of graded commutative topological rings
$$
\lR{\C[dz^i_j,d\bar z^i_j], \, \overline{\C[ \partial_{z^j_1},\cdots, \partial_{z^j_n}]\widetilde\omega_{ij}}} 
\lra H^{\bullet,\bullet}_\bp(\Conf_m(\C^n/\Gamma))
$$
is surjective with kernel generated by holomorphic derivatives of $\widetilde\omega_{ab}\widetilde\omega_{bc}+\widetilde\omega_{bc}\widetilde\omega_{ca}+\widetilde\omega_{ca}\widetilde\omega_{ab}$ and $(d\bar z^a_i-d\bar z^b_i)\widetilde\omega_{ab}$. When modulo kernel, we get an isomorphism of topological rings. 

In other words, $H^{\bullet,\bullet}_\bp(\Conf_m(\C^n/\Gamma))$ is generated by holomorphic functions and holomorphic derivatives of Bochner kernels, subject to the following relations
\begin{itemize}
\item $\widetilde\omega_{ab}=(-1)^n\widetilde\omega_{ba}$
\item $f(\vec \partial_b)\widetilde\omega_{ab} g(\vec \partial_b)\widetilde\omega_{ab}=0$ for holomorphic functions $f,g$, for $n\ge2$
\item (holomorphic derivatives of) $\widetilde\omega_{ab}\widetilde\omega_{bc}+\widetilde\omega_{bc}\widetilde\omega_{ca}+\widetilde\omega_{ca}\widetilde\omega_{ab}=0$
\item $(d\bar z^a_i-d\bar z^b_i)\widetilde\omega_{ab}=0$. 
\end{itemize} 
\end{prop}


\begin{bibdiv}
\begin{biblist}


\bibitem{cohomology-conf} B. Knudsen, {\em Configuration spaces in algebraic topology},  arXiv:1803.11165 (2018).

\bibitem{cohomology-conf-mfd} B. Totaro, {\em Configuration spaces of algebraic varieties},   Topology 35.4 (1996), pp. 1057–1067.

%\bibitem{real-conf-space} F. Cohen, {\em Artin’s braid groups, classical homotopy theory, and sundry other curiosities},  Braids, pp.167-206, Contemporary Mathematics 78, American Mathematical Society, Providence, RI (1988).

\bibitem{Stein} H. Grauert, R. Remmert, {\em Theory of Stein Spaces}, Classics in Mathematics, Springer, 2004.

\bibitem{kevin-owen}
K.~Costello and O.~Gwilliam.
\newblock {\em Factorization algebras in quantum field theory. {V}ol. 1,2},
  volume~31 of {\em New Mathematical Monographs}.
\newblock Cambridge University Press, Cambridge, 2017.

%\bibitem{meng-gluing} L. Meng, {\em Leray–Hirsch theorem and blow-up formula for Dolbeault cohomology}, Annali di Matematica 199, 1997–2014 (2020).

%\bibitem{Gluing-lemma} L. Meng, {\em Morse-Novikov cohomology for blow-ups of complex manifolds}, Pacific J. Math. 320 (2022), no. 2, 365-390.

%\bibitem{MFTF} M. Abate, F. Bracci, T. Suwa, F. Tovena, Localization of Atiyah classes, Rev. Mat. Iberoam. 29 (2013), 547–578.

\bibitem{Griffiths-Harris} P. Griffiths, J. Phillip, {\em Principles of Algebraic Geometry}, Hoboken, N.J. Wiley, 1994.

\bibitem{Bott-Tu} R. Bott, L. Tu,   {\em Differential forms in algebraic topology}, Graduate Texts in Mathematics. Vol. 82. New York-Berlin: Springer-Verlag, 1982.

\bibitem{Suwa-Cech} T. Suwa, {\em Čech-Dolbeault cohomology and the $\bp$-Thom class}, Singularities—Niigata—Toyama 2007, 321–340, Adv. Stud. Pure Math., 56, Math. Soc. Japan, Tokyo, 2009.

%\bibitem{Fulton-Intersection} W. Fulton, {\em Intersection Theory}, Springer-Verlag, 1984.

\end{biblist}
\end{bibdiv}
\end{document}